\documentclass[11pt, reqno]{amsart}

\usepackage{colepaper}

\makeatletter
\def\ps@firstpage{\ps@plain
  \def\@oddfoot{\normalfont\scriptsize \hfil\rule{0pt}{25pt}\thepage\hfil
     \global\topskip\normaltopskip}
  \let\@evenfoot\@oddfoot
  \def\@oddhead{\@serieslogo\hss}
  \let\@evenhead\@oddhead
}
\makeatother
\AtBeginDocument{\thispagestyle{firstpage}}

\newcommand{\ET}{Erd\H{o}s--Tur\'{a}n~}

\author{Cole Graham}
\address{Department of Mathematics, Stanford University, 450 Jane Stanford Way, Building 380, Stanford, CA 94305, USA}
\email{\tt grahamca@stanford.edu}

\title[Irregularity of distribution in Wasserstein distance]{Irregularity of distribution\\in Wasserstein distance}

\begin{document}

\maketitle

\begin{abstract}
  We study the non-uniformity of probability measures on the interval and the circle.
  On the interval, we identify the \mbox{Wasserstein-$p$} distance with the classical $L^p$-discrepancy.
  We thereby derive sharp estimates in Wasserstein distances for the irregularity of distribution of sequences on the interval and the circle.
  Furthermore, we prove an $L^p$-adapted \ET inequality.
\end{abstract}

\section{Introduction}
\label{sec:intro}

We consider the classical question of irregularity of distribution: if we successively place points in a box, how evenly can we space them?
Answers encompass a vast body of theoretical and numerical work.
Rather than cite all related literature, we direct the reader to the excellent survey \cite{Bilyk} and monograph \cite{BC}.

In this note, we restrict our attention to one dimension.
Given a sequence of points $(x_n)_{n \in \N}$ in the interval $I = [0, 1)$ how well can the empirical measures
\begin{equation}
  \label{eq:empirical}
  \mu_N \coloneqq \frac 1 N \sum_{n=1}^N \delta_{x_n}
\end{equation}
approximate the uniform distribution $\lambda$?
Our answer involves several distances $d$ on the space of probability measures $\m P(I)$, but all agree that if the first $N$ points are evenly spaced on $I$,
\begin{equation}
  \label{eq:even-spacing}
  d(\mu_N, \lambda) \sim \frac{C}{N}.
\end{equation}
However, the truncations $(x_n)_{n=1}^N$ of a fixed sequence $(x_n)_{n \in \N}$ cannot \emph{all} be evenly spaced, so we naturally wonder whether \eqref{eq:even-spacing} can hold for \emph{all} $N \in \N$.
Indeed, van der Corput conjectured and van Aardenne-Ehrenfest confirmed that such uniform even spacing is impossible \cite{vdC1, vAE1}.

These classical results concern distances $d$ based on the \emph{discrepancy function}
\begin{equation}
  \label{eq:discrepancy}
  D_N(x) \coloneqq \mu_N([0, x)) - x \quad \textrm{for } x \in I.
\end{equation}
For $p \in [1, \infty]$, the \emph{$L^p$-discrepancy} of a sequence $(x_n)$ at stage $N$ is $\norm{D_N}_{L^p}$.
Thus the $L^\infty$-discrepancy is simply the Kolmogorov--Smirnov statistic \cite{Kolmogorov}.
The following theorem sharply answers van der Corput's conjecture in $L^p$-discrepancy, and unites the work of numerous authors.
To state it cleanly, we let
\begin{equation}
  \al_p \coloneqq
  \begin{cases}
    \frac 1 2 & \textrm{if } p \in [1, \infty),\\
    1 & \textrm{if } p = \infty.
  \end{cases}
\end{equation}

\begin{theorem}[\cite{Roth, Schmidt, Halasz, Lerch, vdC2, Davenport, Chen}]
  \label{thm:disc-irreg}
  For every $p \in [1, \infty]$, there exists a constant $C_p > 0$ such that for any sequence $(x_n)_{n \in \N} \subset I$,
  \begin{equation}
    \norm{D_N}_{L^p} \geq C_p \frac{\log^{\al_p}N}{N}
  \end{equation}
  holds for infinitely many $N \in \N$.
  Furthermore, this bound is sharp.
\end{theorem}

\begin{remark}
  Since $D_N$ is an antiderivative of $\mu_N - \lambda$, we in fact have
  \begin{equation}
    \label{eq:disc=Sob}
    \norm{D_N}_{L^p} = \norm{\mu_N - \lambda}_{\dot W^{-1, p}}.
  \end{equation}
  Thus Theorem~\ref{thm:disc-irreg} quantifies the non-uniformity of $\mu_N$ in the negative Sobolev norm $\norm{\anon}_{\dot W^{-1, p}}$, which we define in Section~\ref{sec:W=S}.
\end{remark}

In this note, we are interested in the distance between $\mu_N$ and $\lambda$ in Wasserstein metrics.
These hail from the venerable theory of optimal transport, but their application to irregularity of distribution appears to be recent \cite{Steinerberger, BS}.
We recall that the Wasserstein-$p$ distance on a metric space $X$ measures the $L^p$-weighted cost of transporting one probability measure to another.
Precisely, for $p \in [1, \infty]$ and $\mu, \nu \in \m P(X)$,
\begin{equation}
  \label{eq:Wasserstein}
  W_p^X(\mu, \nu) \coloneqq \inf_{\pi \in \Pi(\mu, \nu)} \norm{d_X}_{L^p(\pi)}
\end{equation}
where $d_X\colon X \times X \to [0,\infty)$ denotes the metric on $X$, and $\Pi(\mu, \nu)$ denotes the set of \emph{coupling measures} from $\mu$ to $\nu$, i.e. the set of probability measures on $X \times X$ with marginals $\mu$ and $\nu$.

For us, $X$ will be the interval $I = [0, 1)$ or the circle $\T \coloneqq \R/\Z$.
In these one-dimensional spaces, the optimal coupling in \eqref{eq:Wasserstein} assumes a particularly simple form.
This allows us to identify the Wasserstein and Sobolev metrics.
\begin{proposition}
  \label{prop:W=S}
  Let $X$ be $I$ or $\T$, and let $\lambda$ denote the uniform measure on $X$.
  Then for all $p \in [1, \infty]$ and $\mu \in \m P(X)$,
  \begin{equation}
    \label{eq:W=S}
    W_p^X(\mu, \lambda) = \norm{\mu - \lambda}_{\dot W^{-1,p}(X)}.
  \end{equation}
\end{proposition}

\begin{remark}
  This identity extends Kantorovich--Rubinstein duality to $p > 1$, and clarifies a well-known infinitesimal equivalence between $W_2$ and $\dot H^{-1}$ \cite{KR, OV, BB}.
  In this $p = 2$ case, it likewise sharpens a bound of Peyre \cite{Peyre}, who investigated the relationship between $W_2$ and $\dot H^{-1}$ in detail.
  However, these prior results hold in more general settings, while \eqref{eq:W=S} seems restricted to distances from the uniform measure in one dimension.
  We note that on the interval, the $p = 2$ case of \eqref{eq:W=S} appears in \cite[Ex.~64]{Santambrogio}.
\end{remark}

With this identity, we can use classical theory to establish irregularity of distribution in Wasserstein metrics on the interval and circle.
\begin{theorem}
  \label{thm:Wass-irreg}
  Let $X$ be $I$ or $\T$.
  For every $p \in [1, \infty]$ there exists a constant $C_p > 0$ such that for any sequence $(x_n)_{n \in \N} \subset X$,
  \begin{equation}
    W_p^X(\mu_N, \lambda) \geq C_p\frac{\log^{\al_p} N}{N}
  \end{equation}
  holds for infinitely many $N \in \N$.
  Furthermore, this bound is sharp.
\end{theorem}

\begin{remark}
  To our knowledge, irregularity results on the circle have not appeared previously in the literature.
  The distinction between $I$ and $\T$ may seem trivial, but there exist sequences on $I$ that are asymptotically more evenly distributed when viewed on $\T$.
  Indeed, we shall identify $W_2^\T(\mu_N,\lambda)$ with the classical \emph{diaphony}~\cite{Zinterhof, SZ}.
  When $(x_n)$ is the van der Corput sequence~\cite{vdC2}, the asymptotic difference between $W_2^\T(\mu_N,\lambda)$ and $W_2^I(\mu_N,\lambda)$ follows from~\cite{PG}.
\end{remark}

\begin{remark}
  The lower bounds in Theorem~\ref{thm:Wass-irreg} answer several open questions posed by Steinerberger and Brown in \cite{Steinerberger, BS}.
\end{remark}

We close with an explicit \emph{upper} bound on non-uniformity.
The Hausdorff--Young inequality bounds the Sobolev distance $\norm{\mu - \lambda}_{\dot W^{-1,p}(\T)}$ via the Fourier coefficients $\h \mu$ when $p \in [2, \infty]$.
However, we might hope that non-uniformity to scale $\frac 1 n$ could be detected by the first $n$ Fourier coefficients.
When $p = \infty$, this is a consequence of the classical \ET inequality.
We develop an $L^p$ variant for all $p \in [2, \infty]$.
\begin{proposition}
  \label{prop:ET}
  There exists universal $C>0$ such that for all $p \in [2,\infty]$, $n \in \N$, and $\mu \in \m P(\T)$,
  \begin{equation}
    \label{eq:ET-circle}
    \norm{\mu - \lambda}_{\dot W^{-1,p}(\T)} \leq \frac C n + C\left(\sum_{k=1}^{n-1}\frac{\abs{\h \mu(k)}^q}{k^q}\right)^{\frac 1 q},
  \end{equation}
  where $q \in [1, 2]$ is the H\"older-conjugate of $p$.
  Furthermore, let $\smallbar D$ denote the mean of the discrepancy function of $\mu$.
  Then on the interval,
  \begin{equation}
    \label{eq:ET-interval}
    \norm{\mu - \lambda}_{\dot W^{-1,p}(I)} \leq \frac{C}{n} + C\left(\sum_{k=1}^{n-1}\frac{\abs{\h \mu(k)}^q}{k^q} + \abs{\smallbar D}^q\right)^{\frac 1 q}.
  \end{equation}
\end{proposition}

\begin{remark}
  In \cite{Steinerberger}, Steinerberger introduced an \ET inequality for the \mbox{Wasserstein-$1$} distance.
  The $p = \infty$ case of \eqref{eq:ET-circle} shows that in fact the same bound holds for the Wasserstein-$\infty$ distance.
  As we shall argue in Section~\ref{sec:ET}, the $p = \infty$ case is equivalent to the classical \ET inequality.
\end{remark}

We prove Proposition~\ref{prop:W=S} in Section \ref{sec:W=S}, and use it to establish Theorem~\ref{thm:Wass-irreg} in Section~\ref{sec:irreg}.
We prove Proposition~\ref{prop:ET} in Section~\ref{sec:ET}, and apply it to the equidistribution of quadratic residues in finite fields.

\section*{Acknowledgements}

We warmly thank Stefan~Steinerberger and Andrea~Ottolini for their suggestions and encouragement.
This work was supported by the Fannie and John~Hertz Foundation and by NSF grant DGE-1656518.

\section{Wasserstein and Sobolev concur}
\label{sec:W=S}

We define the homogeneous negative Sobolev norm by duality.
For ${q \in [1, \infty]}$ and $X = I$ or $\T$, let
\begin{equation}
  \norm{f}_{\dot W^{1,q}(X)} \coloneqq \|f'\|_{L^q(X)}
\end{equation}
denote the homogeneous Sobolev seminorm of a measurable function
${f \colon X \to \R}$ with weak derivative $f'$.
For H\"older-conjugate $(p, q)$, we define
\begin{equation}
  \label{eq:neg-Sob}
  \norm{\mu - \lambda}_{\dot W^{-1,p}(X)} \coloneqq \sup\left\{\int_X f (\mu - \lambda) \mathrel{\Big|} \norm{f}_{\dot W^{1,q}(X)} \leq 1\right\}.
\end{equation}
Note that the seminorm $\norm{\anon}_{\dot W^{1, q}(X)}$ is invariant under constant shifts of $f$.
Such shifts leave the above expression unchanged, because $\mu - \lambda$ is orthogonal to the constant function.

We first relate our Sobolev norm to the $L^p$-discrepancy.
\begin{lemma}
  \label{lem:antideriv}
  Fix $p \in [1, \infty]$ and $\mu \in \m P(X)$.
  Let $F$ be an antiderivative of $\mu$.
  Then
  \begin{align}
    \label{eq:antideriv-interval}
    \norm{\mu - \lambda}_{\dot W^{-1, p}(I)} &= \norm{F - F(0)}_{L^p(I)}\\
    \intertext{and}
    \label{eq:antideriv-circle}
    \norm{\mu - \lambda}_{\dot W^{-1, p}(\T)} &= \inf_{y \in \R} \norm{F - y}_{L^p(\T)}.
  \end{align}
\end{lemma}

\begin{proof}
  On $I = [0, 1)$, we have
  \begin{equation}
    \int_I f(\mu - \lambda) = \int_I f \d F = - \int_I F \d f + Ff\Big|_0^1.
  \end{equation}
  Since $\mu - \lambda$ is mean zero, $F(0) = F(1)$ and
  \begin{equation}
    \int_I f(\mu - \lambda) = - \int_I [F - F(0)] f'.
  \end{equation}
  Taking the supremum over all $f' \in L^q(I)$, standard $L^p$-duality yields \eqref{eq:antideriv-interval}.

  On $\T$ there are no boundary terms, so $\int_I f(\mu - \lambda) = - \int_I F f'.$
  However, a function in $L^q(\T)$ is a derivative precisely when it has mean zero.
  We thus take the supremum over all $f' \in L_0^q(\T)$, the mean-zero subspace of $L^q(\T)$.
  By subspace-quotient duality,
  \begin{equation}
    \norm{\mu - \lambda}_{\dot W^{-1, p}(\T)} = \norm{F}_{L^p(\T)/\R 1}.
  \end{equation}
  Now \eqref{eq:antideriv-circle} follows from the definition of the quotient norm.
\end{proof}

In particular, Lemma~\ref{lem:antideriv} implies \eqref{eq:disc=Sob}.
We now turn to Wasserstein metrics.
Before proving Proposition~\ref{prop:W=S}, we note that the $p = 1$ case simply restates Kantorovich--Rubinstein duality for our measures \cite{KR}.
Indeed, $\norm{f}_{\dot W^{1,\infty}(X)} = \op{Lip}(f)$, so we can write \eqref{eq:W=S} as
\begin{equation}
  W_1^X(\mu, \lambda) = \sup\left\{\int_X f (\mu - \lambda) \mathrel{\Big|} \op{Lip}(f)\leq 1 \right\}.
\end{equation}

\begin{proof}[Proof of Proposition~\ref{prop:W=S}]
  We handle the interval first, so fix $\mu \in \m P(I)$ and ${p \in [1, \infty)}$.
  It is well-known that the optimal transport map from $\mu$ to $\lambda$ is \emph{monotone}, i.e. preserves the order of the mass \cite[\S 2.2]{Villani}.
  Let $F_\mu$ denote the unique left-continuous antiderivative of $\mu$ such that $F_\mu(0) = 0$, so that $F_\mu(x) = \mu([0,x))$ for $x \in I$.
  Let $F_\mu^{-1}$ denote the left-continuous pseudo-inverse of $F_\mu$:
  \begin{equation}
    F_\mu^{-1}(u) \coloneqq \inf \{x \in I \mid F_\mu(x) \geq u\}.
  \end{equation}
  Then for each $u \in (0, 1]$, the optimal transport plan from $\mu$ to $\lambda$ moves mass at position $F_\mu^{-1}(u)$ to position $u$.
  Thus
  \begin{equation}
    W_p^I(\mu, \lambda)^p = \int_I \abs{F_\mu^{-1}(u) - u}^p \ds u.
  \end{equation}
  We now appeal to a curious identity from single-variable calculus, which we state in more general terms in anticipation of the circle case.
  \begin{lemma}
    \label{lem:headturn}
    Let $h \colon \R \to \R$ be left-continuous and non-decreasing such that $h - \op{id}$ is $1$-periodic.
    Let $h^{-1}$ denote its left-continuous pseudo-inverse.
    Then for continuous $\phi \colon \R \to \R$,
    \begin{equation}
      \label{eq:headturn}
      \int_I \phi(h(x) - x) \d x = \int_I \phi(u - h^{-1}(u)) \d u.
    \end{equation}
  \end{lemma}

  \begin{proof}
    We can approximate $h$ and $h^{-1}$ pointwise almost-everywhere by $h_n$ and $h_n^{-1}$, respectively, for some sequence of $\m C^1$ diffeomorphisms $(h_n)$.
    We may thus assume that $h$ is itself a $\m C^1$ diffeomorphism.

    Let $\m I_1$ and $\m I_2$ denote the left- and right-hand sides of \eqref{eq:headturn}, respectively.
    With the change of variables $u = h(x)$, we have
    \begin{equation}
      \label{eq:I2-shifted}
      \m I_2 = \int_{h^{-1}(I)} \phi(h(x) - x) h'(x) \d x.
    \end{equation}
    Now $h - \op{id}$ is $1$-periodic, so $h^{-1}(I)$ is an interval of length 1.
    Since $h'$ is also $1$-periodic, we can shift the region of integration in \eqref{eq:I2-shifted} to $I$ without changing the value of the integral.
    If $\psi$ denotes an anti-derivative of $\phi$, we find
    \begin{equation}
      \m I_2 - \m I_1 = \int_I \phi(h(x) - x)[h'(x) - 1] \d x = \int_I \op{d}[\psi(h - \op{id})] = \psi(h - \op{id})\Big|_0^1 = 0.
    \end{equation}
    The final equality follows from $1$-periodicity.
  \end{proof}
  On the interval, we can restrict the domain of $h$ to $I$.
  Taking $h = F_\mu^{-1}$ and $\phi(z) = \abs{z}^p$, we find
  \begin{equation}
    W_p^I(\mu, \lambda)^p = \int_I \abs{F_\mu^{-1}(u) - u}^p \ds u = \int_I \abs{F_\mu(x) - x}^p \ds x.
  \end{equation}
  Now $F \coloneqq F_\mu - \op{id}$ is an antiderivative of $\mu - \lambda$ with $F(0) = 0$, so by Lemma~\ref{lem:antideriv},
  \begin{equation}
    W_p^I(\mu, \lambda) = \norm{F}_{L^p(I)} = \norm{\mu - \lambda}_{\dot W^{-1, p}(I)}.
  \end{equation}

  We next modify the argument to work on the torus.
  Given $\mu \in \m P(\T)$, we can lift $\mu$ and $\lambda$ to 1-periodic measures on the universal cover $\R$.
  These measures have infinite mass on $\R$, so we search for transport plans which are optimal with respect to any \emph{local} modification.
  We still expect the monotone transport plans to be locally optimal, but these are no longer unique.
  Indeed, there is ambiguity in where we begin ``filling-in'' the mass of $\mu$ with that of $\lambda$.
  That is, to transport $\lambda$ to $\mu$ we could first move mass at 0 to position $y$, and then fill-in monotonically around that starting point.
  We call this the $y$-monotonic transport plan.
  
  In \cite{DSS}, Delon, Salomon, and Sobolevski confirm that the locally optimal transport plans are precisely the $y$-monotonic plans.
  Furthermore, an optimal plan on $\T$ lifts to a locally optimal plan on $\R$.
  Thus an optimal transport plan on $\T$ is the restriction of a $y$-monotonic plan to $I$ for some $y\in \R$.
  
  To determine \emph{which} $y$ is optimal on $\T$, we compute the transport cost $C_p[y]$ in $I$ under the $y$-monotonic plan.
  With $F_\mu$ defined as before, let $F_\mu^y\coloneqq F_\mu(\,\cdot \, - y)$.
  Then the cost is
  \begin{equation}
    C_p[y]^p = \int_I \abs{(F_\mu^y)^{-1}(u) - u}^p \ds u
  \end{equation}
  and \cite{DSS} shows that
  \begin{equation}
    W_p^\T(\mu,\lambda) = \inf_{y\in \R} C_p[y].
  \end{equation}

  Applying Lemma~\ref{lem:headturn} with $h = F_\mu^y$ and $\phi(z) = \abs{z}^p$, we find
  \begin{equation}
    W_p^\T(\mu,\lambda)^p = \inf_{y\in \R} \int_I \abs{(F_\mu^y)^{-1}(u) - u}^p \ds u = \inf_{y\in \R} \int_I \abs{F_\mu^y(x) - x}^p \ds x.
  \end{equation}
  Shifting by $y$, periodicity and Lemma~\ref{lem:antideriv} imply
  \begin{equation}
    W_p^\T(\mu,\lambda) = \inf_{y\in \R} \norm{F - y}_{L^p(\T)} = \norm{\mu - \lambda}_{\dot W^{-1, p}(\T)}.
  \end{equation}
  
  We have thus verified \eqref{eq:W=S} on $I$ and $\T$ for all $p \in [1, \infty)$.
  Taking $p \to \infty$, we automatically obtain \eqref{eq:W=S} at the endpoint $p = \infty$.
  Nevertheless, we offer a self-contained proof in this case.
  
  Assume $\mu \neq \lambda$.
  For $A\subset X$ and $\eps>0$, let
  \begin{equation}
    A^\eps \coloneqq \{x\in X\;|\; \op{dist}(x,A) \leq \eps\}
  \end{equation}
  denote the $\eps$-neighborhood of $A$.
  As a consequence of Strassen's theorem \cite{GS, Strassen},
  \begin{equation}
    \label{eq:Strassen}
    W_\infty^X(\mu,\lambda) = \inf\left\{\eps \geq 0\;|\;\mu(A) \leq \lambda(A^\eps) \text{ for all Borel } A\subset X\right\}.
  \end{equation}
  Let $\eps_0 \coloneqq \norm{\mu - \lambda}_{\dot W^{-1,\infty}(X)} > 0$, and fix a Borel set $A\subset X$.
  Let
  \begin{equation}
    A^*\coloneqq \{x\in A^{\eps_0}\;|\;\op{dist}(x, \partial A^{\eps_0}) \geq \eps_0\}.
  \end{equation}
  Clearly, $A \subset A^*$.
  Also, $A^{\eps_0}$ and $A^*$ are unions of closed intervals, and each point in $\partial A^*$ corresponds to an $\eps_0$-wide buffer of $A^{\eps_0}$ around $A^*$.
  (Note that on $I$ we use the subspace topology, so $\{0,1\}$ are never boundary points of closed sets.)
  Thus
  \begin{equation}
    \mu(A^*) - \lambda(A^*) = \int_X \tbf{1}_{A^*} (\mu - \lambda) \leq \norm{\tbf{1}_{A^*}'}_{\mathrm{TV}}\norm{\mu - \lambda}_{\dot W^{-1,\infty}(X)} = \abs{\partial A^*} \eps_0
  \end{equation}
  and
  \begin{equation}
    \mu(A) \leq \mu(A^*) \leq \abs{\partial A^*} \eps_0 + \lambda(A^*) = \lambda(A^{\eps_0}\setminus A^*) + \lambda(A^*) = \lambda(A^{\eps_0}).
  \end{equation}
  By \eqref{eq:Strassen}, we have $W_\infty^X(\mu,\lambda)\leq \eps_0 = \norm{\mu - \lambda}_{\dot W^{-1,\infty}(X)}.$

  For the other direction, fix $\eps \in (0, \eps_0)$ and let $\delta \coloneqq \eps_0 - \eps$.
  First consider $X = \T$, and let $F$ denote a left-continuous anti-derivative of $\mu - \lambda$.
  By Lemma~\ref{lem:antideriv},
  \begin{equation}
    \norm{\mu - \lambda}_{\dot W^{-1,\infty}(\T)} = \inf_{y\in \R} \norm{F - y}_{L^\infty(\T)}.
  \end{equation}
  Evidently, this infimum is attained at $y = \frac{\sup F + \inf F}{2}$.
  Thus
  \begin{equation}
    \eps_0 = \norm{\mu - \lambda}_{\dot W^{-1,\infty}(\T)} = \frac{\sup F - \inf F}{2}.
  \end{equation}
  Take $x_\pm\in \T$ such that
  \begin{equation}
    F(x_+) > \sup F - \delta \quad \text{and}\quad F(x_-) < \inf F + \delta.
  \end{equation}
  After a rotation, we may assume that $0 \leq x_- \leq x_+ < 1$ in our representation $\T = \R/\Z$.
  Let $J = [x_-, x_+)$.
  Then
  \begin{equation}
    \mu(J) - \lambda(J) = F(x_+) - F(x_-) > \sup F - \inf F - 2\delta = 2\eps_0 - 2\delta = 2\eps.
  \end{equation}
  Hence
  \begin{equation}
    \mu(J) > \lambda(J) + 2\eps \geq \lambda(J^\eps).
  \end{equation}
  Since such an interval exists for all $\eps < \eps_0$, \eqref{eq:Strassen} implies
  \begin{equation}
    W_\infty^\T(\mu,\lambda) \geq \eps_0 = \norm{\mu - \lambda}_{\dot W^{-1,\infty}(\T)}.
  \end{equation}

  When $X = I$, we let $F$ denote the left-continuous antiderivative such that $F(0) = 0$, so that $\norm{\mu - \lambda}_{\dot W^{-1,\infty}(I)} = \norm{F}_{L^\infty(I)}.$
  If we take $J = [0, x_+)$ or $[x_-, 1)$, manipulations much as above yield
  \begin{equation}
    \mu(J) > \lambda(J) + \eps \geq \lambda(J^\eps),
  \end{equation}
  and the conclusion follows as before.
\end{proof}

\section{Irregularity of distribution in one dimension}
\label{sec:irreg}

We now pivot to the irregularity of distribution of sequences.
Since we have identified the $L^p$-discrepancy with the Sobolev and Wasserstein distances, Theorem~\ref{thm:Wass-irreg} largely follows from the classical results in Theorem~\ref{thm:disc-irreg}.
The principal remaining obstacle is the set of lower bounds in Theorem~\ref{thm:Wass-irreg} on the circle.
To prove these lower bounds, we follow the harmonic analysis approach of Roth and Hal\'asz \cite{Roth, Halasz}.
Bilyk lucidly presents this method in~\cite{Bilyk}, which we follow closely.

In his seminal paper \cite{Roth}, Roth observed a crucial equivalence between irregularity problems on $I$ and $I^2$.
Roth proved that point distributions on $I^2$ are irregular \emph{for all} $N \in \N$, and showed that this implies irregularity in one dimension for infinitely many $N$.
We adapt this approach to distributions on the circle.

Given a finite point set $\{\tbf x_n\}_{n=1}^N \subset \T \times I$, let $\mu$ denote its empirical measure as in \eqref{eq:empirical}.
We define the two-dimensional discrepancy function $\m D \colon \T \times I \to \R$ by
\begin{equation}
  \label{eq:disc-2D}
  \m D(x, y) \coloneqq \mu([0, x) \times [0, y)) - xy.
\end{equation}
By Proposition~\ref{prop:W=S} and Lemma~\ref{lem:antideriv}, we wish to control the one-dimensional discrepancy $D_N$ in the quotient space $L^p(\T)/\R1$.
We must therefore measure $\m D$ in a space that is compatible with these quotients on horizontal slices.
Let $V^p \coloneqq 1 \otimes L^p(I)$ denote the closed subspace of $L^p(\T \times I)$ consisting of functions that are independent of the first variable.
We will control $\m D$ in the quotient space $L^p(\T \times I)/V^p$.
The stratified structure of this space permits us to descend to the quotient $L^p(\T)/\R1$ on horizontal slices.
\begin{lemma}
  \label{lem:irreg-cyl}
  There exists a universal constant $C > 0$ such that for any ${N \in \N}$, $p \in [1, \infty]$, and point set $\{\tbf x_n\}_{n=1}^N \subset \T \times I$, the discrepancy function $\m D$ of the point set satisfies
  \begin{equation}
    \label{eq:irreg-cyl}
    \norm{\m D}_{L^p(\T \times I)/V^p} \geq C \frac{\log^{\al_p}N}{N}.
  \end{equation}
\end{lemma}
Before proving the lemma, we use it to establish our main theorem.
\begin{proof}[Proof of Theorem~\ref{thm:Wass-irreg}.]
  The identities \eqref{eq:disc=Sob} and \eqref{eq:W=S} imply
  \begin{equation}
    \norm{D_N}_{L^p(I)} = W_p^I(\mu_N, \lambda).
  \end{equation}
  Thus on the interval, Theorem~\ref{thm:Wass-irreg} simply restates the classical results in Theorem~\ref{thm:disc-irreg}.
  Furthermore, any transport plan on $I$ can be interpreted as a plan on $\T$, so $W_p^\T \leq W_p^I$.
  It follows that the upper bounds in Theorem~\ref{thm:disc-irreg} on $I$ also serve as upper bounds in Theorem~\ref{thm:Wass-irreg} on $\T$.
  See also \cite{Steinerberger} for an upper bound proven specifically for the Wasserstein-$2$ distance on $\T$.
  It remains only to prove the lower bounds on $\T$.

  Fix $p \in [1, \infty]$ and a sequence $(x_n)_{n \in \N} \subset \T$.
  For each $m \in \N$, we form an $m$-point set $\{\tbf x_n^m\}_{n=1}^m \subset \T \times I$ by
  \begin{equation}
    \tbf x_n^m \coloneqq \left(x_n, \frac{n-1}{m}\right).
  \end{equation}
  That is, we embed the first $m$ terms of $(x_n)$ in the cylinder by interpreting the ``time'' $\frac{n-1}{m}$ as a second coordinate.
  Let $\m D^m$ denote the two-dimensional discrepancy function associated to this $m$-point set by \eqref{eq:disc-2D}.
  By Lemma~\ref{lem:irreg-cyl},
  \begin{equation}
    \norm{\m D^m}_{L^p(\T \times I)/V^p} \geq C \frac{\log^{\al_p} m}{m}.
  \end{equation}
  That is,
  \begin{equation}
    \inf_{g \in L^p(I)} \norm{\m D^m - 1 \otimes g}_{L^p(\T \times I)} \geq  C\frac{\log^{\al_p} m}{m}.
  \end{equation}
  By Fubini, there exists $t_m \in I$ such that
  \begin{equation}
    \inf_{y \in \R} \norm{\m D^m(\anon, t_m) - y}_{L^p(\T)} \geq C \frac{\log^{\al_p} m}{m}.
  \end{equation}
  Now $t_m$ is nearly the time corresponding to the index $N_m \coloneqq \ceil{m t_m}$.
  Writing $t_m = \frac{N_m}{m} - \delta$ with $\delta \in \left[0, \frac 1 m\right)$, we have
  \begin{equation}
    \m D^m(x, t_m) = \frac{N_m}{m} D_{N_m}(x) - \delta x \quad \textrm{for all } x \in \T,
  \end{equation}
  where $D_{N_m}$ is the one-dimensional discrepancy of $(x_n)$ at stage $N_m$.
  Therefore
  \begin{align}
    \norm{D_{N_m}}_{L^p(\T)/\R 1} &\geq \frac{m}{N_m} \left(\norm{\m D^m(\anon, t_m)}_{L^p(\T)/\R 1} - \delta \norm{\id{}}_{L^p(\T)}\right)\\
                                  &\geq \frac{C}{N_m}\left(\log^{\al_p}m - C'\right).
  \end{align}
  Once $m$ is sufficiently large, we find
  \begin{equation}
    \label{eq:final-bd}
    \norm{D_{N_m}}_{L^p(\T)/\R 1} \geq C\frac{\log^{\al_p} m}{N_m} \geq C\frac{\log^{\al_p} N_m}{N_m},
  \end{equation}
  where we have allowed the constant $C > 0$ to change from line to line.
  The first inequality in \eqref{eq:final-bd} cannot hold as $m \to \infty$ if the sequence $(N_m)_{m \in \N}$ remains bounded, so it must hold for infinitely many $N_m$.
\end{proof}

\begin{remark}
  In fact, the $p = 2$ case of Theorem~\ref{thm:Wass-irreg} follows from classical results.
  Indeed, in this case Proposition~\ref{prop:W=S} and Parseval imply
  \begin{equation}
    \label{eq:Parseval}
    W_2^\T(\mu_N, \lambda) = \norm{\mu_N - \lambda}_{\dot H^{-1}(\T)} = \frac 1 {2\pi} \left(\sum_{k \in \Z \setminus \{0\}} \frac{\abs{\h \mu_N(k)}^2}{k^2}\right)^{\frac 1 2}.
  \end{equation}
  The final expression is the diaphony of Zinterhof (up to a factor of $2\pi$).
  In \cite{Proinov}, Pro\u{\nakedi}nov announced that the diaphony exceeds $C\frac{\sqrt{\log N}}{N}$ infinitely often, confirming Theorem~\ref{thm:disc-irreg} in the case $p = 2$.
  (We were, however, unable to locate a published proof.)
  
  Additionally, if $F$ denotes the antiderivative of $\mu - \lambda$ with $F(0) = 0$ and $\smallbar F \coloneqq \int_I F$, Proposition~\ref{prop:W=S} and Lemma~\ref{lem:antideriv} imply
  \begin{equation}
    W_2^I(\mu_N, \lambda)^2 = W_2^\T(\mu_N, \lambda)^2 + {\smallbar F}^2.
  \end{equation}
  This is simply a reinterpretation of a formula of Koksma \cite{K1, K2}.
  It precisely quantifies the additional transport cost incurred when the movement of mass across the point $0 \in \T$ is forbidden.
\end{remark}

\begin{proof}[Proof of Lemma~\ref{lem:irreg-cyl}]
  Fix $N \in \N$, the point set ${\{\tbf x_n\}_{n=1}^N \subset \T \times I}$, and its discrepancy function $\m D$.
  In \cite{Bilyk}, Bilyk condenses ideas of Roth and Hal\'asz to prove
  \begin{equation}
    \label{eq:2D-baby}
    \norm{\m D}_{L^p(\T \times I)} \geq C \frac{\log^{\al(p)}N}{N}
  \end{equation}
  for any $p \in [1, \infty]$.
  (Note that our discrepancy function differs from that in \cite{Bilyk} by a factor of $N^{-1}$.)
  This suffices for irregularity on the interval.
  On the torus, however, we must measure $\m D$ in the smaller quotient norm of $L^p/V^p$.
  We claim that Bilyk actually proves the stronger bound \eqref{eq:irreg-cyl}.
  
  When $p = 2$, Bilyk follows Roth and proves \eqref{eq:2D-baby} via duality \cite[\S 2]{Bilyk}.
  He uses Haar wavelets to construct an explicit test function $F$ on $\T \times I$ such that
  \begin{equation}
    \norm{F}_{L^2(\T \times I)} \asymp \sqrt{\log N} \And \braket{\m D, F}_{L^2(\T \times I)} \gtrsim \frac{\log N}{N}.
  \end{equation}
  The estimate \eqref{eq:2D-baby} then follows from Cauchy-Schwarz.
  To control $\m D$ in $L^2/V^2$, we must handle coset representatives of the form $\m D - 1 \otimes g$ for $g \in L^2(I)$.
  Fortuitously, the Haar wavelets comprising $F$ are mean-zero \emph{along every horizontal slice}.
  It follows that
  \begin{equation}
    \braket{F, 1 \otimes g} = 0
  \end{equation}
  for all $g \in L^p(I)$.
  Therefore
  \begin{equation}
    \norm{\m D - 1 \otimes g}_{L^2} \geq \frac{\braket{\m D - 1 \otimes g, F}}{\norm{F}_2} = \frac{\braket{\m D, F}}{\norm{F}_2} \gtrsim \frac{\sqrt{\log N}}{N}.
  \end{equation}
  Since this holds uniformly for $g \in L^2(I)$, we obtain
  \begin{equation}
    \norm{\m D}_{L^2/V^2} \gtrsim \frac{\sqrt{\log N}}{N},
  \end{equation}
  i.e. the $p = 2$ case of Lemma~\ref{lem:irreg-cyl}.

  When $p \in (1, \infty)$, Bilyk employs Littlewood--Paley theory \cite[\S 3]{Bilyk} to show that $\norm{F}_q \asymp \sqrt{\log N}$, where $q$ denotes the H\"{o}lder exponent dual to $p$.
  We thus obtain Lemma~\ref{lem:irreg-cyl} for $p \in (1, \infty)$ in an identical manner.
  When $p = 1$ or $\infty$, Hal\'{a}sz had the insight of using ``Riesz products'' as test functions \cite{Halasz}.
  The construction is slightly more elaborate, but in \cite[\S 4]{Bilyk} we nonetheless arrive at test functions in the span of the Haar basis which are mean-zero on horizontal slices.
  These test functions are thus also immune to the quotient space complication, and imply the $p = 1$ and $\infty$ cases of Lemma~\ref{lem:irreg-cyl}.
\end{proof}

\section{An $L^p$ \ET inequality}
\label{sec:ET}

The classical \ET inequality \cite{ET1, ET2} controls another notion of discrepancy on the circle:
\begin{equation}
  \op{disc}(\mu, \lambda)\coloneqq \sup_{J\subset \T}\abs{\mu(J) - \lambda(J)},
\end{equation}
where the supremum ranges over all subintervals $J \subset \T$.
It states:
\begin{equation}
  \label{eq:ET}
  \op{disc}(\mu, \lambda) \leq \frac{C}{n} + C \sum_{k=1}^{n-1}\frac{\abs{\h \mu(k)}}{k} \quad \text{for all }n\in \N.
\end{equation}
In fact, this form of discrepancy is cleanly related to those we've already considered:
\begin{equation}
  \label{eq:ET-disc-Sob}
  \op{disc}(\mu, \lambda) = 2 \norm{\mu - \lambda}_{\dot W^{-1,\infty}(\T)} = 2 W_\infty^\T(\mu, \lambda).
\end{equation}
After all, if $F$ denotes a left-continuous antiderivative of $\mu - \lambda$, the regularity of $\mu$ implies
\begin{equation}
  \sup_{J\subset \T}\abs{\mu(J) - \lambda(J)} = \sup F - \inf F = 2 \norm{\mu - \lambda}_{\dot W^{-1,\infty}(\T)}.
\end{equation}
By \eqref{eq:ET-disc-Sob}, the $p = \infty$ case of Proposition~\ref{prop:ET} is equivalent to \eqref{eq:ET}.

To prove Proposition~\ref{prop:ET}, we introduce a lemma in the mode of Ganelius \cite{Ganelius}.
For a real function $V\in L^\infty(\T)$, define the one-sided modulus of continuity
\begin{equation}
  \omega(\delta;V)\coloneqq \sup_{s\in [t,t+\delta]}[V(s) - V(t)].
\end{equation}

\begin{lemma}
  \label{lem:Ganelius}
  There exists a universal $C>0$ such that for all $p\in [2,\infty]$, $n\in \N$, and $V \in L^\infty(\T)$,
  \begin{equation}
    \norm{V}_{L^p(\T)} \leq C \, \omega\left(\frac C n ; V\right) + C\left(\sum_{k=0}^{n-1} |\h V(k)|^q\right)^{\frac 1 q},
  \end{equation}
  where $q$ is the H\"older-conjugate of $p$.
\end{lemma}

Proposition~\ref{prop:ET} follows easily from Lemma~\ref{lem:Ganelius}:
\begin{proof}[Proof of Proposition~\ref{prop:ET}.]
  Let $F$ denote the mean-zero antiderivative of $\lambda - \mu$ (opposite our earlier choice).
  By Lemma~\ref{lem:antideriv}, $\norm{\mu - \lambda}_{\dot{W}^{-1,p}(\T)} \leq \norm{F}_{L^p(\T)}.$
  Now note that
  \begin{equation}
    \omega(\delta; F) = \sup_{s\in [t,t+\delta]}\left[(s-t) - \mu([t,s))\right] \leq \delta.
  \end{equation}
  In particular, when $\delta = \frac C n$, Lemma~\ref{lem:Ganelius} implies
  \begin{equation}
    \label{eq:pre-ET}
    \norm{F}_{L^p(\T)} \leq \frac C n + C\left(\sum_{k=0}^{n-1}|\h F(k)|^q\right)^{\frac 1 q},
  \end{equation}
  where we have allowed $C$ to change from line to line.
  Now $\h F(0) = 0$ because $F$ is mean-zero, and $\h F(k) = \frac{\h \mu(k)}{2\pi \i k}$ for $k \geq 1$.
  Thus \eqref{eq:pre-ET} implies \eqref{eq:ET-circle}.

  Now consider the interval.
  The discrepancy function ${D(x) \coloneqq \mu([0,x)) - x}$ is a left-continuous antiderivative of $\mu - \lambda$ with $D(0) = 0$.
  By Lemma~\ref{lem:antideriv}, $\norm{\mu - \lambda}_{\dot{W}^{-1,p}(I)} = \norm{D}_{L^p(I)}$.
  Now $\h D(0) = \int_I D \eqqcolon \smallbar D$.
  Using $D$ in the place of $F$ in \eqref{eq:pre-ET} and retaining the $k = 0$ term in the sum, we obtain \eqref{eq:ET-interval}.
\end{proof}

To prove the lemma, we follow the approach of Ganelius \cite{Ganelius}.
\begin{proof}[Proof of Lemma \ref{lem:Ganelius}.]
  Fix $n \in \N$ and $V \in L^\infty(\T)$.
  We first assume $p\in [2,\infty)$.
  Let
  \begin{equation}
    \Psi_n(t) \coloneqq \frac{\sin^2(\pi nt)}{n\sin^2(\pi t)}
  \end{equation}
  denote the Fej\'er kernel.
  Then $\Psi_n$ has mass $1$ concentrated at scale $\frac 1 n$ around $t = 0$.
  Precisely, there exists $C_0>0$ such that for all $\delta\in (0, 1)$,
  \begin{equation}
    \int_{-\delta}^\delta \Psi_n(t) \d t \geq 1 - \frac{C_0}{n\delta}.
  \end{equation}
  Fix $\delta \coloneqq \frac{8C_0}{n},$ so that
  \begin{equation}
    \int_{\T\setminus [-\delta,\delta]}\Psi_n \leq \frac 1 {8}.
  \end{equation}
  Let $\omega\coloneqq \omega(2\delta; V)$.
  Define the subsets
  \begin{equation}
    \Omega^\pm \coloneqq \{t\in \T\;|\; \pm V(t) > \omega\}.
  \end{equation}
  By the definition of $\omega$, the translated sets $\Omega^+ - \delta$ and $\Omega^- + \delta$ are disjoint.

  Next, let $\Psi_n^1\coloneqq \tbf{1}_{[-\delta,\delta]}\Psi_n$ and $\Psi_n^2\coloneqq \Psi_n - \Psi_n^1$.
  By the generalized AM-GM inequality, $\abs{a + b}^p \leq 2^{p-1}(\abs{a}^p + \abs{b}^p)$ for all $a,b\in \R$.
  After substitution and rearrangement, this implies
  \begin{equation}
    \label{eq:arithmetic}
    \abs{a - b}^p \geq 2^{1 - p}\abs{a}^p - \abs{b}^p\quad\text{for all }a,b \in \R.
  \end{equation}
  In particular,
  \begin{equation}
    \label{eq:split}
    \int_\T \abs{\Psi_n\ast V}^p \geq 2^{1-p}\int_\T \abs{\Psi_n^1\ast V}^p - \int_\T \abs{\Psi_n^2\ast V}^p.
  \end{equation}
  We control the second term with Young's inequality:
  \begin{equation}
    \norm{\Psi_n^2\ast V}_{L^p}^p \leq \norm{\Psi_n^2}_{L^1}^p\norm{V}_{L^p}^p \leq 2^{-3p}\norm{V}_{L^p}^p.
  \end{equation}
  For the first term in \eqref{eq:split}, write
  \begin{equation}
    \int_\T \abs{\Psi_n^1\ast V}^p \geq \int_{\Omega^+ - \delta} \abs{\Psi_n^1\ast V}^p + \int_{\Omega^- + \delta} \abs{\Psi_n^1\ast V}^p.
  \end{equation}
  Now if $t\in \Omega^+$, the definition of $\omega$ implies
  \begin{equation}
    \Psi_n^1\ast V(t - \delta) = \int_{-\delta}^\delta \Psi_n(s)V(t - \delta - s) \d s \geq [V(t) - \omega]\int_{-\delta}^\delta \Psi_n(s) \d s \geq \frac 1 2 [V(t) - \omega].
  \end{equation}
  Similarly, if $t\in \Omega^-$, we have
  \begin{equation}
    \abs{\Psi_n^1\ast V(t + \delta)} \geq \frac 1 2 [\abs{V(t)} - \omega].
  \end{equation}
  Thus by \eqref{eq:arithmetic} again,
  \begin{equation}
    \int_{\Omega^\pm \mp \delta} \abs{\Psi_n^1\ast V}^p \geq 2^{-p}\int_{\Omega^\pm}(\abs V - \omega)^p \geq 2^{-p}\int_{\Omega^\pm} \left(2^{1-p}\abs{V}^p - \omega^p\right).
  \end{equation}
  On the other hand, $\abs{V}\leq \omega$ on $\T \setminus (\Omega^+ \cup \Omega^-)$, so
  \begin{equation}
    \int_{\T \setminus (\Omega^+ \cup \Omega^-)} \abs{V}^p \leq \omega^p.
  \end{equation}
  Collecting these results, \eqref{eq:split} implies
  \begin{equation}
    \int_\T \abs{\Psi_n\ast V}^p \geq 2^{2-3p}\int_\T \abs{V}^p - 2^{2-2p}\omega^p - 2^{-3p} \int_\T \abs{V}^p \geq 2^{1-3p}\int_\T \abs{V}^p - 2^{2-2p}\omega^p.
  \end{equation}
  Rearranging, we find
  \begin{equation}
    \int_\T \abs{V}^p \leq 2^{2p}\omega^p + 2^{3p}\int_\T \abs{\Psi_n\ast V}^p.
  \end{equation}
  Taking $\frac 1 p$ powers, this becomes
  \begin{equation}
    \label{eq:F_est}
    \norm{V}_{L^p} \leq 4\omega + 8 \norm{\Psi_n\ast V}_{L^p}.
  \end{equation}
  The same holds for $p = \infty$ by taking limits, or by trivially adjusting the above manipulations.
  In fact, \eqref{eq:F_est} holds for all $p\geq 1$.

  Now recall that
  \begin{equation}
    \h \Psi_n(k) = \left(1 - \frac{\abs{k}}{n}\right)\tbf{1}_{[-n,n]}(k) \leq \tbf{1}_{[1-n, \, n-1]}(k)
  \end{equation}
  When $p\geq 2$, the Hausdorff--Young inequality implies
  \begin{equation}
    \norm{\Psi_n\ast V}_{L^p} \leq \left(\sum_{k=1-n}^{n-1} |\h V(k)|^q\right)^{\frac 1 q}.
  \end{equation}
  Since $V$ is real, we obtain
  \begin{equation}
    \norm{V}_{L^p} \leq C\omega\left(\frac{C}{n};V\right) + C\left(\sum_{k=0}^{n-1}|\h V(k)|^q\right)^{\frac 1 q},
  \end{equation}
  perhaps after increasing $C$.
\end{proof}

\subsection{An application in number theory}

As an example, we use Propositions~\ref{prop:W=S} and \ref{prop:ET} to study the distribution of quadratic residues in finite fields of prime order.
Given a prime $p$, we consider the set
\begin{equation}
  Q_p \coloneqq \left\{ \left\{\frac{m^2}{p}\right\} \mathrel{\Big |} 1 \leq m \leq p\right\} \subset \T,
\end{equation}
where $\{\anon\}$ denotes the fractional part.
We measure the equidistribution of $Q_p$ in Wasserstein-$r$ distances.
We mimic the approach of Steinerberger~\cite{Steinerberger}, who considered $r \in \{2, \infty\}$.
We extend his results to the intermediate exponents $r \in (2, \infty)$.

Let $\mu_p$ denote the empirical measure of $Q_p$ as in \eqref{eq:empirical}.
To apply Proposition~\ref{prop:ET}, we must control the Gauss sum
\begin{equation}
  \h \mu_p(k) = \frac 1 p \sum_{m=1}^p \e^{-2\pi \i k m^2/p}.
\end{equation}
This is a classical object in analytic number theory.
Gauss showed:
\begin{equation}
  \label{eq:Gauss}
  \abs{\h \mu_p(k)} =
  \begin{cases}
    1 & \textrm{if } p \,|\, k, \\
    \frac 1 {\sqrt{p}} & \textrm{if } p \nmid k.
  \end{cases}
\end{equation}

Now let $(r, s)$ be H\"older-conjugates with $r \in [2, \infty]$.
By Proposition~\ref{prop:ET} with $n = p$, we have
\begin{equation}
  W_r^\T(\mu_p, \lambda) \leq \frac{C}{p} + C \left(\sum_{k=1}^{p-1} \frac{\abs{\h \mu_p(k)}^s}{k^s}\right)^{\frac 1 s} \leq \frac{C}{p} + \frac{C}{\sqrt{p}}\left(\sum_{k=1}^{p-1} k^{-s}\right)^{\frac 1 s}.
\end{equation}
Thus when $r \in [2, \infty)$, there exists a constant $C_r > 0$ such that
\begin{equation}
  \label{eq:residue-finite-r}
  W_r^\T(\mu_p, \lambda) \leq \frac{C_r}{\sqrt{p}}.
\end{equation}
In fact, this holds for all $r \in [1, \infty)$ by H\"older's inequality.
However, when $r = \infty$ we lose a logarithmic factor:
\begin{equation}
  \label{eq:residue-infinity}
  W_\infty^\T(\mu_p, \lambda) \leq \frac{C_\infty \log p}{\sqrt{p}}.
\end{equation}
In light of \eqref{eq:ET-disc-Sob}, we have thus recovered a well-known result of P\'olya and Vinogradov \cite{Polya, Vinogradov}.

Establishing lower bounds seems more challenging, but we can handle the case $p = 2$ explicitly.
Indeed, \eqref{eq:Parseval} and \eqref{eq:Gauss} imply
\begin{equation}
  W_2^\T(\mu_p, \lambda)^2 = \frac 1 {4\pi^2} \sum_{k \in \Z \setminus \{0\}} \frac{\abs{\h \mu_p(k)}^2}{k^2} = \frac 1 {4 \pi^2}\left(\sum_{k \in p\Z \setminus \{0\}} \frac 1 {k^2} + \frac 1 p \sum_{k \in \Z \setminus p \Z} \frac 1 {k^2}\right).
\end{equation}
Rearranging these sums, we obtain
\begin{equation}
  W_2^\T(\mu_p, \lambda) = \sqrt{\frac{p^2 + p - 1}{12p^3}} \geq \frac{1}{\sqrt{12 p}}.
\end{equation}
It follows that \eqref{eq:residue-finite-r} is tight up to the value of $C_r$ for all $r \in [2, \infty)$.
However, assuming the generalized Riemann hypothesis, Montgomery and Vaughan~\cite{MV} showed that \eqref{eq:residue-infinity} can be improved to
\begin{equation}
  W_\infty^\T(\mu_p, \lambda) \leq \frac{C \log \log p}{\sqrt{p}}.
\end{equation}
A construction of Paley \cite{Paley} suggests this is optimal up to the value of $C$.

\bibliographystyle{colestyle}
\bibliography{Wasserstein-irregularity-arXiv}

\end{document}